\documentclass{amsart}
\usepackage{amssymb}
\usepackage{amsthm}
\usepackage{comment}
\usepackage{tikz}

\newtheorem{thm}{Theorem}
\newtheorem{lemma}{Lemma}
\newtheorem*{rmk}{Remark}

\renewcommand{\mod}[1]{\ (\text{mod}\ #1)}

\title{On a Divisor of the Central Binomial Coefficient}
\author{Matthew Just and Maxwell Schneider}
\date{}

\begin{document}

\maketitle

\begin{abstract}
   It is well known that for all $n\geq 1$ the number $n+1$ is a divisor of the central binomial coefficient ${2n\choose n}$. Since the $n$th central binomial coefficient equals the number of lattice paths from $(0,0)$ to $(n,n)$ by unit steps north or east, a natural question is whether there is a way to partition these paths into sets of $n+1$ paths or $n+1$ equinumerous sets of paths. The Chung--Feller theorem gives an elegant answer to this question. We pose and deliver an answer to the analogous question for $2n-1$, another divisor of ${2n\choose n}$. We then show our main result follows from a more general observation regarding binomial coefficients ${n\choose k}$ with $n$ and $k$ relatively prime. A discussion of the case where $n$ and $k$ are not relatively prime is also given, highlighting the limitations of our methods. Finally, we come full circle and give a novel interpretation of the Catalan numbers.
\end{abstract}


\section{Introduction}

The central binomial coefficients, defined for $n\geq 1$ by \[{2n\choose n} = \frac{(2n)!}{(n!)^2},\] have many far reaching applications. Perhaps the most famous is Erd\H{o}s' elementary proof of Bertrand's postulate \cite{erdos1932beweis}. From a combinatorial perspective, the central binomial coefficient is equal to the number of lattice paths from $(0,0)$ to $(n,n)$ by taking one of the two steps in the set $\{ (1,0),(0,1)\}$. We will refer to these paths as \textit{northeastern lattice paths}, or simply lattice paths. One such path is illustrated in Figure \ref{northeasternlatticepath} for $n=4$.

\begin{figure}[h]
\begin{center}
    \begin{tikzpicture}[scale=.5]
        \draw[step=1,thin] (0,0) node[anchor=north east] {} grid (4,4)node[anchor=south west] {};
        \draw[ultra thick,red] (0,0) -- (1,0) -- (1,3)--(3,3)--(3,4)--(4,4);
    \end{tikzpicture}
\end{center}
\caption{A northeastern lattice path from $(0,0)$ to $(4,4)$ formed by taking unit steps north or east. The number of such paths is ${8\choose 4} = 70$.}
\label{northeasternlatticepath}
\end{figure}
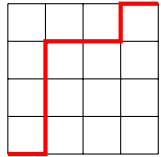

It was conjectured by Erd\H{o}s that ${2n\choose n}$ is squarefree for all $n>4$. This conjecture was proved for every sufficiently large $n$ by S\'ark\"ozy \cite{sarkozy1985divisors}, and subsequently proved for $n>4$ by Granville and Ramar\'e \cite{granville1996explicit}.

One look at the identity \[\frac{1}{n+1} {2n\choose n} = {2n\choose n} - {2n\choose n+1}\] shows that $(n+1)\mid {2n\choose n}$ for all $n\geq 0$. The numbers \[C_n = \frac{1}{n+1} {2n\choose n}\] defined for $n\geq 0$ are referred to as the Catalan numbers \cite{stanley2011enumerative}. Though described by Euler and Catalan in relation to parenthesized expressions and dissection of polygons, the Catalan numbers can be traced back to the Mongolian mathematician Minggatu, who used them to express certain trigonometric series. A detailed historical account is given by Larcombe \cite{larcombe199918th}, who cites a 1988 paper of Luo \cite{luo1988first} giving Minggatu the rightful credit. 

In the context of northeastern lattice paths, the $n$th Catalan number equals the number of paths which lie below the main diagonal. Such paths are referred to as \textit{Dyck paths}. One such path is illustrated in Figure \ref{belowdiagonal}, again for $n=4$.

\begin{figure}[h]
\begin{center}
    \begin{tikzpicture}[scale=.5]
        \draw[step=1,thin] (0,0) node[anchor=north east] {} grid (4,4)node[anchor=south west] {};
        \draw[thick] (0,0) -- (4,4);
        \draw[ultra thick,red] (0,0) -- (2,0) -- (2,2)--(3,2)--(3,3)--(4,3)--(4,4);
    \end{tikzpicture}
\end{center}
\caption{A northeastern lattice path from $(0,0)$ to $(4,4)$ for which all steps lie below the main diagonal. The number of such paths is given by the Catalan number $C_4=14$.}
\label{belowdiagonal}
\end{figure}
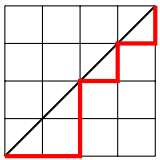

While this gives a combinatorial proof that $(n+1)\mid {2n\choose n}$, it does not fully explain how to partition the set of all ${2n\choose n}$ paths into subsets of equal cardinality. Observe that the number of vertical steps $(0,1)$ that can be taken above the main diagonal can be any number between 0 and $n$. The Chung--Feller theorem \cite{feller2015fluctuations}, originally stated in terms of coin flips, implies that for any $j$, $0\leq j \leq n$, the number of paths that have $j$ vertical steps above the main diagonal is independent of $j$, implying there are $C_n$ paths in each one of these $n+1$ equinumerous sets of paths. An example of this partitioning is shown in Figure \ref{excedence}.
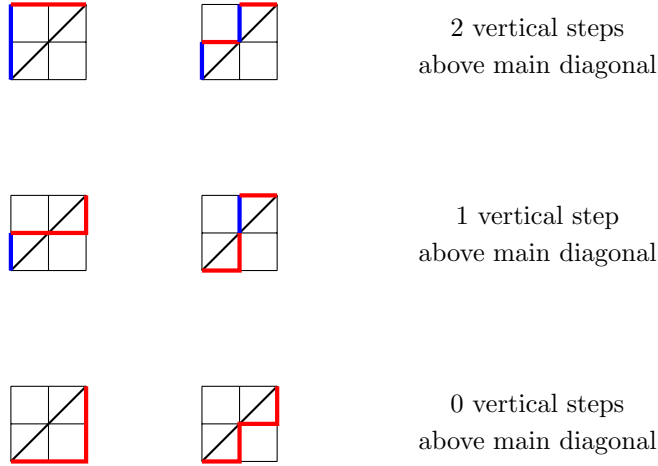
\begin{figure}[h]
\begin{center}
    \begin{tikzpicture}[scale=.5]
        \begin{scope}
            \draw[step=1,thin] (0,0) node[anchor=north east] {} grid (2,2)node[anchor=south west] {};
        \draw[thick] (0,0) -- (2,2);
        \draw[ultra thick,red] (0,0) -- (2,0) -- (2,2);
        \end{scope}
        \begin{scope}[xshift=2in]
            \draw[step=1,thin] (0,0) node[anchor=north east] {} grid (2,2)node[anchor=south west] {};
        \draw[thick] (0,0) -- (2,2);
        \draw[ultra thick,red] (0,0) -- (1,0) -- (1,1)--(2,1)--(2,2);
        \end{scope}
        \begin{scope}[yshift=2in]
            \draw[step=1,thin] (0,0) node[anchor=north east] {} grid (2,2)node[anchor=south west] {};
        \draw[thick] (0,0) -- (2,2);
        \draw[ultra thick,blue] (0,0) -- (0,1);
        \draw[ultra thick, red] (0,1) -- (2,1)--(2,2);
        \end{scope}
        \begin{scope}[xshift=2in,yshift=2in]
            \draw[step=1,thin] (0,0) node[anchor=north east] {} grid (2,2)node[anchor=south west] {};
        \draw[thick] (0,0) -- (2,2);
        \draw[ultra thick,red] (0,0) -- (1,0) -- (1,1);
        \draw[ultra thick,blue](1,1)--(1,2);
        \draw[ultra thick,red](1,2)--(2,2);
        \end{scope}
        \begin{scope}[yshift=4in]
            \draw[step=1,thin] (0,0) node[anchor=north east] {} grid (2,2)node[anchor=south west] {};
        \draw[thick] (0,0) -- (2,2);
        \draw[ultra thick,blue] (0,0)--(0,2);
        \draw[ultra thick,red] (0,2) -- (2,2);
        \end{scope}
        \begin{scope}[xshift=2in, yshift=4in]
            \draw[step=1,thin] (0,0) node[anchor=north east] {} grid (2,2)node[anchor=south west] {};
        \draw[thick] (0,0) -- (2,2);
        \draw[ultra thick,blue] (0,0) -- (0,1);
        \draw[ultra thick,red] (0,1) -- (1,1);
        \draw[ultra thick,blue] (1,1) -- (1,2);
        \draw[ultra thick,red] (1,2) -- (2,2);
        \end{scope}
        \node[align=center] at (14,1.5){0 vertical steps};
        \node[align=center] at (14,.5){above main diagonal};
        \node[align=center] at (14,6.5){1 vertical step};
        \node[align=center] at (14,5.5){above main diagonal};
        \node[align=center] at (14,11.5){2 vertical steps};
        \node[align=center] at (14,10.5){above main diagonal};
    \end{tikzpicture}
\end{center}
\caption{The ${4\choose 2}=6$ northeastern lattice paths from $(0,0)$ to $(2,2)$ partitioned into equinumerous sets by the number of northern steps taken above the main diagonal. }
\label{excedence}
\end{figure}
Labelle and Yeh \cite{labelle1990generalized} give a generalized proof of this result.

A lesser known divisor of ${2n\choose n}$ is $2n-1$, for all $n\geq 1$. Similarly to $n+1$, this is not obvious until you take one glance at the carefully crafted identity 
\begin{align}
    {2n\choose n} &= 2\cdot (2n-1) \cdot \frac{1}{n} {2(n-1)\choose n-1}. \label{keyid}
\end{align}
Note we are using the fact that the Catalan number $C_{n-1}$ is an integer.  For any fixed $k$, the set of natural numbers $n$ such that
$n + k$ divides ${2n\choose n}$ has been studied by Pomerance \cite{pomerance2015divisors}. Further improvements in this direction are given by Sanna \cite{sanna2018central}, as well as Ford and Konyagin \cite{ford2020divisibility}.

Motivated by the Chung--Feller theorem we ask the question as to whether there is a partitioning of the ${2n\choose n}$ lattice paths into $2n-1$ sets with the same cardinality. Our study proved fruitful while investigating the area enclosed by the lattice paths. Figure \ref{area} shows an example of this enclosed area by a path.

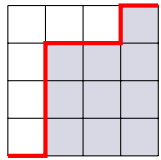
\begin{figure}[h]
\begin{center}
    \begin{tikzpicture}[scale=.5]
        \fill[blue!50!olive,opacity=.2] (1,0)--(1,3)--(3,3)--(3,4)--(4,4)--(4,0)--cycle;
        \draw[step=1,thin] (0,0) node[anchor=north east] {} grid (4,4)node[anchor=south west] {};
        \draw[ultra thick,red] (0,0) -- (1,0) -- (1,3)--(3,3)--(3,4)--(4,4);
    \end{tikzpicture}
\end{center}
\caption{A northeastern lattice path from $(0,0)$ to $(4,4)$ enclosing an area of 10.}
\label{area}
\end{figure}

Computational data suggested to the authors that if we took each northeastern lattice path from $(0,0)$ to $(n,n)$ and reduced the enclosed area modulo $2n-1$, the ${2n\choose n}$ paths were partitioned equally amongst the $2n-1$ residue classes. Figure \ref{part} shows this observation for $n=3$.

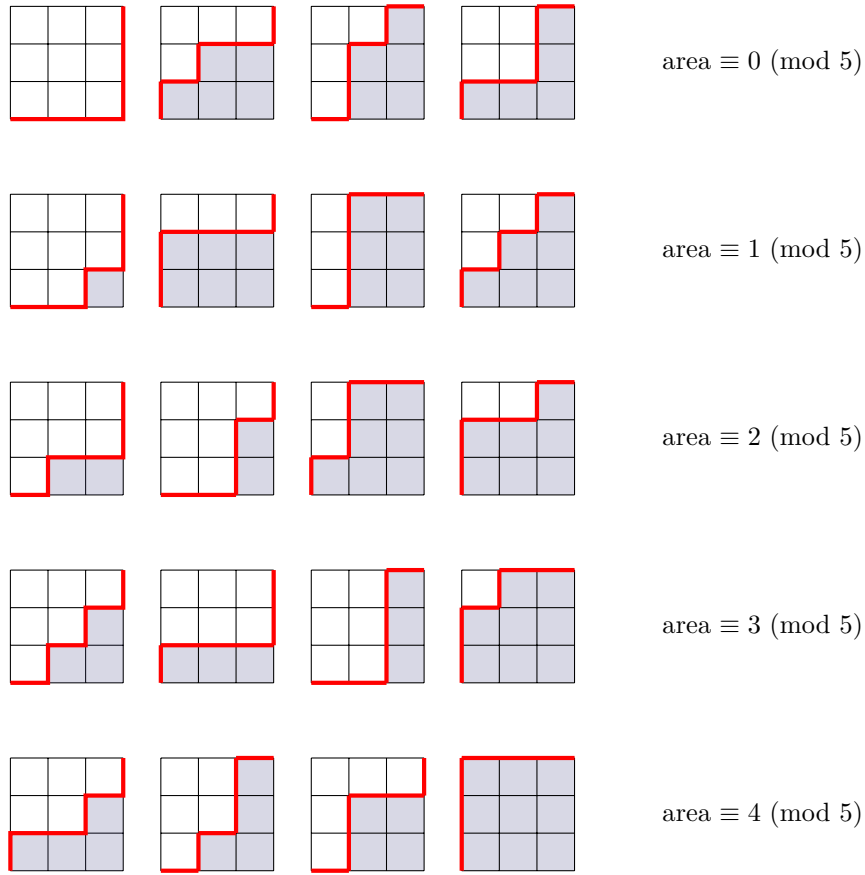
\begin{figure}[h]
\begin{center}
    \begin{tikzpicture}[scale=.5]
        \begin{scope}[yshift=20cm]
            \draw[step=1,thin] (0,0) node[anchor=north east] {} grid (3,3)node[anchor=south west] {};
        
        \draw[ultra thick,red] (0,0) -- (3,0) -- (3,3);
        \end{scope}
        \begin{scope}[xshift=4cm,yshift=20cm]
        \fill[blue!50!olive,opacity=.2] (0,0)--(0,1)--(1,1)--(1,2)--(3,2)--(3,0)--cycle;
            \draw[step=1,thin] (0,0) node[anchor=north east] {} grid (3,3)node[anchor=south west] {};
        
        \draw[ultra thick,red] (0,0) -- (0,1);
        \draw[ultra thick,red] (0,1) -- (1,1);
        \draw[ultra thick,red] (1,1) -- (1,2);
        \draw[ultra thick,red] (1,2) -- (3,2);
        \draw[ultra thick,red] (3,2) -- (3,3);
        \end{scope}
        \begin{scope}[xshift=8cm,yshift=20cm]
        \fill[blue!50!olive,opacity=.2] (0,0)--(1,0)--(1,2)--(2,2)--(2,3)--(3,3)--(3,0)--cycle;
            \draw[step=1,thin] (0,0) node[anchor=north east] {} grid (3,3)node[anchor=south west] {};
        
        \draw[ultra thick,red] (0,0) -- (1,0);
        \draw[ultra thick,red] (1,0) -- (1,2);
        \draw[ultra thick,red] (1,2) -- (2,2);
        \draw[ultra thick,red] (2,2) -- (2,3);
        \draw[ultra thick,red] (2,3) -- (3,3);
        \end{scope}
        \begin{scope}[xshift=12cm,yshift=20cm]
        \fill[blue!50!olive,opacity=.2] (0,0)--(0,1)--(2,1)--(2,3)--(3,3)--(3,0)--cycle;
            \draw[step=1,thin] (0,0) node[anchor=north east] {} grid (3,3)node[anchor=south west] {};
        
        \draw[ultra thick,red] (0,0) -- (0,1);
        \draw[ultra thick,red] (0,1) -- (2,1);
        \draw[ultra thick,red] (2,1) -- (2,3);
        \draw[ultra thick,red] (2,3) -- (3,3);
        \end{scope}
        
        \begin{scope}[yshift=15cm]
            \draw[step=1,thin] (0,0) node[anchor=north east] {} grid (3,3)node[anchor=south west] {};
            \fill[blue!50!olive,opacity=.2] (2,0)--(2,1)--(3,1)--(3,0)--cycle;
        
        \draw[ultra thick,red] (0,0) -- (2,0) -- (2,1)--(3,1)--(3,3);
        \end{scope}
        \begin{scope}[xshift=4cm,yshift=15cm]
        \fill[blue!50!olive,opacity=.2] (0,0)--(0,2)--(3,2)--(3,0)--cycle;
            \draw[step=1,thin] (0,0) node[anchor=north east] {} grid (3,3)node[anchor=south west] {};
        
        \draw[ultra thick,red] (0,0) -- (0,2);
        \draw[ultra thick,red] (0,2) -- (3,2);
        \draw[ultra thick,red] (3,2) -- (3,3);
        \end{scope}
        \begin{scope}[xshift=8cm,yshift=15cm]
        \fill[blue!50!olive,opacity=.2] (0,0)--(1,0)--(1,3)--(3,3)--(3,0)--cycle;
            \draw[step=1,thin] (0,0) node[anchor=north east] {} grid (3,3)node[anchor=south west] {};
        
        \draw[ultra thick,red] (0,0) -- (1,0);
        \draw[ultra thick,red] (1,0) -- (1,3);
        \draw[ultra thick,red] (1,3) -- (3,3);
        \end{scope}
        \begin{scope}[xshift=12cm,yshift=15cm]
        \fill[blue!50!olive,opacity=.2] (0,0)--(0,1)--(1,1)--(1,2)--(2,2)--(2,3)--(3,3)--(3,0)--cycle;
            \draw[step=1,thin] (0,0) node[anchor=north east] {} grid (3,3)node[anchor=south west] {};
        
        \draw[ultra thick,red] (0,0) -- (0,1);
        \draw[ultra thick,red] (0,1) -- (1,1);
        \draw[ultra thick,red] (1,1) -- (1,2);
        \draw[ultra thick,red] (1,2) -- (2,2);
        \draw[ultra thick,red] (2,2) -- (2,3);
        \draw[ultra thick,red] (2,3) -- (3,3);
        \end{scope}
        
        \begin{scope}[yshift=10cm]
            \draw[step=1,thin] (0,0) node[anchor=north east] {} grid (3,3)node[anchor=south west] {};
            \fill[blue!50!olive,opacity=.2] (0,0)--(1,0)--(1,1)--(3,1)--(3,0)--cycle;
        
        \draw[ultra thick,red] (0,0) -- (1,0) -- (1,1)--(3,1)--(3,3);
        \end{scope}
        \begin{scope}[xshift=4cm,yshift=10cm]
        \fill[blue!50!olive,opacity=.2] (0,0)--(2,0)--(2,2)--(3,2)--(3,0)--cycle;
            \draw[step=1,thin] (0,0) node[anchor=north east] {} grid (3,3)node[anchor=south west] {};
        
        \draw[ultra thick,red] (0,0) -- (2,0);
        \draw[ultra thick,red] (2,0) -- (2,2);
        \draw[ultra thick,red] (2,2) -- (3,2);
        \draw[ultra thick,red] (3,2) -- (3,3);
        \end{scope}
        \begin{scope}[xshift=8cm,yshift=10cm]
        \fill[blue!50!olive,opacity=.2] (0,0)--(0,1)--(1,1)--(1,3)--(3,3)--(3,0)--cycle;
            \draw[step=1,thin] (0,0) node[anchor=north east] {} grid (3,3)node[anchor=south west] {};
        
        \draw[ultra thick,red] (0,0) -- (0,1);
        \draw[ultra thick,red] (0,1) -- (1,1);
        \draw[ultra thick,red] (1,1) -- (1,3);
        \draw[ultra thick,red] (1,3) -- (3,3);
        \end{scope}
        \begin{scope}[xshift=12cm,yshift=10cm]
        \fill[blue!50!olive,opacity=.2] (0,0)--(0,2)--(2,2)--(2,3)--(3,3)--(3,0)--cycle;
            \draw[step=1,thin] (0,0) node[anchor=north east] {} grid (3,3)node[anchor=south west] {};
        
        \draw[ultra thick,red] (0,0) -- (0,2);
        \draw[ultra thick,red] (0,2) -- (2,2);
        \draw[ultra thick,red] (2,2) -- (2,3);
        \draw[ultra thick,red] (2,3) -- (3,3);
        \end{scope}
        
        \begin{scope}[yshift=5cm]
            \draw[step=1,thin] (0,0) node[anchor=north east] {} grid (3,3)node[anchor=south west] {};
             \fill[blue!50!olive,opacity=.2] (1,0)--(1,1)--(2,1)--(2,2)--(3,2)--(3,3)--(3,0)--cycle;
        
        \draw[ultra thick,red] (0,0) -- (1,0) -- (1,1)--(2,1)--(2,2)--(3,2)--(3,3);
        \end{scope}
        \begin{scope}[xshift=4cm,yshift=5cm]
        \fill[blue!50!olive,opacity=.2] (0,0)--(0,1)--(3,1)--(3,0)--cycle;
            \draw[step=1,thin] (0,0) node[anchor=north east] {} grid (3,3)node[anchor=south west] {};
        
        \draw[ultra thick,red] (0,0) -- (0,1);
        \draw[ultra thick,red] (0,1) -- (3,1);
        \draw[ultra thick,red] (3,1) -- (3,3);
        \end{scope}
        \begin{scope}[xshift=8cm,yshift=5cm]
        \fill[blue!50!olive,opacity=.2] (2,0)--(2,3)--(3,3)--(3,0)--cycle;
            \draw[step=1,thin] (0,0) node[anchor=north east] {} grid (3,3)node[anchor=south west] {};
        
        \draw[ultra thick,red] (0,0) -- (2,0);
        \draw[ultra thick,red] (2,0) -- (2,3);
        \draw[ultra thick,red] (2,3) -- (3,3);
        \end{scope}
        \begin{scope}[xshift=12cm,yshift=5cm]
        \fill[blue!50!olive,opacity=.2] (0,0)--(0,2)--(1,2)--(1,3)--(3,3)--(3,0)--cycle;
            \draw[step=1,thin] (0,0) node[anchor=north east] {} grid (3,3)node[anchor=south west] {};
        
        \draw[ultra thick,red] (0,0) -- (0,2);
        \draw[ultra thick,red] (0,2) -- (1,2);
        \draw[ultra thick,red] (1,2) -- (1,3);
        \draw[ultra thick,red] (1,3) -- (3,3);
        \end{scope}
        
        \begin{scope}[yshift=0cm]
            \draw[step=1,thin] (0,0) node[anchor=north east] {} grid (3,3)node[anchor=south west] {};
            \fill[blue!50!olive,opacity=.2] (0,0)--(0,1)--(2,1)--(2,2)--(3,2)--(3,0)--cycle;
        
        \draw[ultra thick,red] (0,0) -- (0,1) -- (2,1)--(2,2)--(3,2)--(3,3);
        \end{scope}
        \begin{scope}[xshift=4cm,yshift=0cm]
        \fill[blue!50!olive,opacity=.2] (0,0)--(1,0)--(1,1)--(2,1)--(2,3)--(3,3)--(3,0)--cycle;
            \draw[step=1,thin] (0,0) node[anchor=north east] {} grid (3,3)node[anchor=south west] {};
        
        \draw[ultra thick,red] (0,0) -- (1,0);
        \draw[ultra thick,red] (1,0) -- (1,1);
        \draw[ultra thick,red] (1,1) -- (2,1);
        \draw[ultra thick,red] (2,1) -- (2,3);
        \draw[ultra thick,red] (2,3) -- (3,3);
        \end{scope}
        \begin{scope}[xshift=8cm,yshift=0cm]
        \fill[blue!50!olive,opacity=.2] (0,0)--(1,0)--(1,2)--(3,2)--(3,3)--(3,0)--cycle;
            \draw[step=1,thin] (0,0) node[anchor=north east] {} grid (3,3)node[anchor=south west] {};
        
        \draw[ultra thick,red] (0,0) -- (1,0);
        \draw[ultra thick,red] (1,0) -- (1,2);
        \draw[ultra thick,red] (1,2) -- (3,2);
        \draw[ultra thick,red] (3,2) -- (3,3);
        \end{scope}
        \begin{scope}[xshift=12cm,yshift=0cm]
        \fill[blue!50!olive,opacity=.2] (0,0)--(0,3)--(3,3)--(3,0)--cycle;
            \draw[step=1,thin] (0,0) node[anchor=north east] {} grid (3,3)node[anchor=south west] {};
        
        \draw[ultra thick,red] (0,0) -- (0,3);
        \draw[ultra thick,red] (0,3) -- (3,3);
        \end{scope}

        \node[align=center] at (20,21.5){area $\equiv 0\mod{5}$};
        \node[align=center] at (20,16.5){area $\equiv 1\mod{5}$};
        \node[align=center] at (20,11.5){area $\equiv 2\mod{5}$};
        \node[align=center] at (20,6.5){area $\equiv 3\mod{5}$};
        \node[align=center] at (20,1.5){area $\equiv 4\mod{5}$};
    \end{tikzpicture}
\end{center}
\caption{The ${6\choose 3}=20$ northeastern lattice paths from $(0,0)$ to $(3,3)$ partitioned by the area they enclose modulo 5.}
\label{part}
\end{figure}

The proof of this observation in general encompasses our main result.

\begin{thm}
    Let $n\geq 1$ and $0\leq j < 2n-1$. Then the number of northeastern lattice paths from $(0,0)$ to $(n,n)$ enclosing an area congruent to $j$ modulo $2n-1$ is equal to $\frac{1}{2n-1}{2n\choose n}$, which in particular is independent from $j$.
\end{thm}

We will give two proofs; in Section 2 we will use $q$-analogues of the binomial coefficients and in Section 3  we will give a combinatorial proof. We then follow up with several equivalent versions of Theorem 1 in Section 3. We end by giving some generalizations of Theorem 1 in Section 4, including applications to products of integers and the Catalan numbers.

\section{Proof of Theorem 1}

We break the proof into three parts. First we will discuss the key lemma involving polynomial division, then we will give a brief introduction to $q$-analogues of numbers, and finally we will put the pieces together.

\subsection{Key lemma}

We start with a polynomial $f(q)$ with integer coefficients, i.e. an element of the polynomial ring $\mathbb{Z}[q]$ \[f(q) = \sum_{i=0}^n a_i q^i.\]
Now for any $m\geq 1$ and $0\leq k <m$ we define the following coefficient sum of $f$
\[b(k,m;f)=\sum_{i\equiv k(m)} a_i\] where the summation is taken over those $0\leq i \leq n$ such that $i\equiv k\mod{m}$. We will then say that $f$ has \textit{equal content modulo $m$} if \[b(k,m;f)=\frac{f(1)}{m}\] for all $k$. The following lemma gives an equivalent characterization for a polynomial $f$ to have equal content modulo $m$ in terms of divisibility of $f$ by a special polynomial. We will denote by $[m]_q$ the polynomial $1+q+q^2+\ldots+q^{m-1}$, often referred to as the $q$-analogue of $m$.

\begin{lemma}
    Let $f(q)$ be a polynomial with integer coefficients and $m\geq1$. Then $f(q)$ has equal content modulo $m$ if and only if the polynomial $[m]_q$ divides $f(q)$ as elements of $\mathbb{Z}[q]$.
\end{lemma}

\begin{proof}
    First assume that $f(q)$ has equal content modulo $m$. It is possible (and an interesting linear algebra exercise) to explicitly construct the polynomial $g(q)$ with integer coefficients such that $f(q)=[m]_q\cdot g(q)$. A slicker proof is to show that each root of $[m]_q$ is also a root of $f(q)$. Suppose that $\alpha$ is a root of $[m]_q$, so that \[1+\alpha+\alpha^2+\ldots + \alpha^{m-1} = 0.\] Moreover, we know that $\alpha$ must be an $m$th root of unity, so that $\alpha^{\ell m + k}=\alpha^k$ for any $\ell \geq 0$.
    
    Now we can rewrite $f(\alpha)$ in the following way, noting that any auxiliary $a_i$ are zero for $i>n$:
    \begin{align*}
        f(\alpha)&=\sum_{i\geq 0} a_i \alpha^i = \sum_{k=0}^{m-1} \sum_{\ell \geq 0} a_{\ell m + k} \alpha^{\ell m + k} = \sum_{k=0}^{m-1} \sum_{\ell \geq 0} a_{\ell m+k} \alpha^k \\
        &= \sum_{k=0}^{m-1}b(k,m;f) \alpha^k  = \frac{f(1)}{m}  \sum_{k=0}^{m-1} \alpha^k = 0.
    \end{align*}
    Since any root of $[m]_q$ is a root of $f(q)$ and $[m]_q$ has no multiple roots, it follows that $[m]_q$ divides $f(q)$.
    
    Conversely, suppose that $[m]_q$ divides $f(q)$. Then there is a polynomial \[g(q)=\sum_{j=0}^{n-m+1}c_j q^j\] such that $f(q)=[m]_q\cdot g(q)$. But if we expand this product (again using auxiliary coefficients that are zero) we have
    \begin{align*}
        [m]_q\cdot g(q) &= \sum_{k=0}^{m-1}q^k \cdot \sum_{j\geq 0 } c_j q^j \\
        &= \sum_{j\geq 0} c_j \sum_{k=0}^{m-1} q^{j+k}.
    \end{align*}
    But as $k$ ranges from zero to $m-1$, $j+k$ also ranges over this same interval modulo $m$. It then follows that the coefficient sum \[b(k,m;f)=\sum_{j\geq 0 }c_j = g(1),\] from which the result follows. \end{proof}
    
    \subsection{$q$-analogues}
    
    We saw in the proof of the Lemma 1 the so-called $q$-analogue of a number $m\geq 1$, given by $[m]_q=1+q+q^2+\ldots+q^{m-1}$. The value of this polynomial at $q=1$ is $[m]_1=m$. Furthermore, define $[m]_q!=\prod_{i=1}^m [i]_q$ to be the $q$-analogue of the factorial. In general, a $q$-analogue of some number is a function (usually a function that equals a polynomial for $|q|<1$) whose limiting value as $q\rightarrow 1^-$ is equal to the value the function is an analogue of. The coefficients of the polynomial then give a combinatorially interesting partition of the number. 
    
    We will be concerned with $q$-analogues of the binomial coefficients---sometimes called Gaussian binomial coefficients---defined by \[{n\brack k}_{q}: =\frac{ [n]_q!}{[k]_q![n-k]_q!}= \frac{\prod_{i=1}^n (q^i-1)}{\prod_{i=1}^{k}(q^i-1)\prod_{i=1}^{n-k}(q^i-1)}\] which, for $|q|<1$, can be shown to be a polynomial. Often this is done by induction. It is also easily verified that the limiting value as $q\rightarrow 1^-$ is the regular binomial coefficient ${n\choose k}$. What will be significant for us will be the interpretation of the coefficient of $q^j$ with regards to the counting northeastern lattice paths. For northeastern lattice paths between $(0,0)$ and $(k,n-k)$ the coefficient of $q^j$ in the polynomial ${n\brack k}_q$ gives the number of the ${n\choose k}$ paths that enclose an area of $j$ (see Figure \ref{qan} for an example). This fact can be deduced from Proposition 1.7.3 in Chapter 1 of Stanley \cite{stanley2011enumerative}. With this interpretation we can also derive the $q$-analogue of Pascal's identity, \begin{align}\label{qpas}
    {n\brack k}_q &=  q^{n-k}{n-1\brack k-1}_q+{n-1\brack k}_q.
    \end{align}
    
    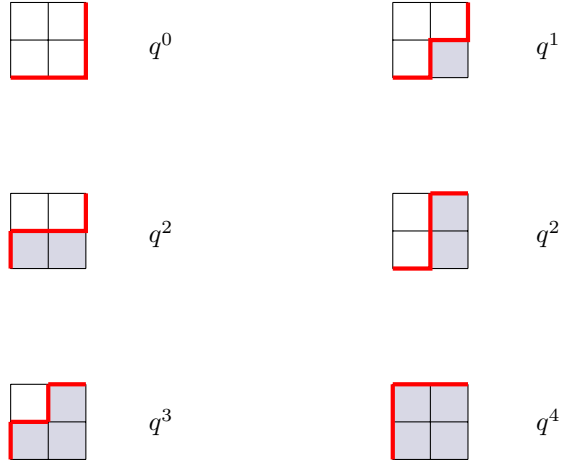
\begin{figure}
    \begin{center}
    \begin{tikzpicture}[scale=.5]
        \begin{scope}[yshift=4in]
            \draw[step=1,thin] (0,0) node[anchor=north east] {} grid (2,2)node[anchor=south west] {};
        
        \draw[ultra thick,red] (0,0) -- (2,0) -- (2,2);
        \end{scope}
        \begin{scope}[xshift=4in,yshift=4in]
        \fill[blue!50!olive,opacity=.2] (1,0) rectangle (2,1);
            \draw[step=1,thin] (0,0) node[anchor=north east] {} grid (2,2)node[anchor=south west] {};
        
        \draw[ultra thick,red] (0,0) -- (1,0) -- (1,1)--(2,1)--(2,2);
        \end{scope}
        \begin{scope}[yshift=2in]
        \fill[blue!50!olive,opacity=.2] (0,0) rectangle (2,1);
            \draw[step=1,thin] (0,0) node[anchor=north east] {} grid (2,2)node[anchor=south west] {};
        
        \draw[ultra thick,red] (0,0) -- (0,1);
        \draw[ultra thick, red] (0,1) -- (2,1)--(2,2);
        \end{scope}
        \begin{scope}[xshift=4in,yshift=2in]
        \fill[blue!50!olive,opacity=.2] (1,0) rectangle (2,2);
            \draw[step=1,thin] (0,0) node[anchor=north east] {} grid (2,2)node[anchor=south west] {};
        
        \draw[ultra thick,red] (0,0) -- (1,0) -- (1,1);
        \draw[ultra thick,red](1,1)--(1,2);
        \draw[ultra thick,red](1,2)--(2,2);
        \end{scope}
        \begin{scope}[yshift=0in,xshift=4in]
        \fill[blue!50!olive,opacity=.2] (0,0) rectangle (2,2);
            \draw[step=1,thin] (0,0) node[anchor=north east] {} grid (2,2)node[anchor=south west] {};
        
        \draw[ultra thick,red] (0,0)--(0,2);
        \draw[ultra thick,red] (0,2) -- (2,2);
        \end{scope}
        \begin{scope}[xshift=0in]
        \fill[blue!50!olive,opacity=.2] (0,0)--(0,1)--(1,1)--(1,2)--(2,2)--(2,0)--cycle;
            \draw[step=1,thin] (0,0) node[anchor=north east] {} grid (2,2)node[anchor=south west] {};
        
        \draw[ultra thick,red] (0,0) -- (0,1);
        \draw[ultra thick,red] (0,1) -- (1,1);
        \draw[ultra thick,red] (1,1) -- (1,2);
        \draw[ultra thick,red] (1,2) -- (2,2);
        \end{scope}
        \node[align=center] at (4,1){$q^3$};
        \node[align=center] at (4,6){$q^2$};
        \node[align=center] at (4,11){$q^0$};
        \node[align=center] at (14.3,1){$q^4$};
        \node[align=center] at (14.3,6){$q^2$};
        \node[align=center] at (14.3,11){$q^1$};
    \end{tikzpicture}
\end{center}
\caption{A visualization of the ${4\choose 2}=6$ northeastern lattice paths enclosing an area equal to $j$ as represented by the coefficients of the polynomial ${4\brack 2}_q = 1+q^2+2q^2+q^3+q^4$.}
\label{qan}
\end{figure}

\subsection{Putting it all together}

Based on our discussion of $q$-analogues we must show that the polynomial ${2n\brack n}_q$ has equal content modulo $2n-1$. By Lemma 1 it then suffices to show that the polynomial $[2n-1]_q$ divides ${2n \brack n}_q$ for all $n\geq 1$. For this we derive a $q$-analogue of equation (\ref{keyid}). By repeatedly applying the $q$-analogue of Pascal's identity (\ref{qpas}) we have
\begin{align*}
    {2n\brack n}_q &= (1+q^n){2n-1\brack n}_q \\
    &= (1+q^n) \left( q^{n-1} {2n-2 \brack n-1}_q + {2n-2 \brack n}_q \right) \\
    &=(1+q^n) \left(q^{n-1} + \frac{1-q^{n-1}}{1-q^n}\right) {2(n-1)\brack n-1}_q \\
    &=(1+q^n)\cdot [2n-1]_q \cdot \frac{{2(n-1) \brack n-1}_q}{[n]_q}.
\end{align*}
We now just need to show that \[\frac{{2(n-1) \brack n-1}_q}{[n]_q}\] is a polynomial for all $n\geq 1$. To see why this is the case one can verify the identity
\begin{align}
    \label{catd} \frac{1}{[n+1]_q}{2n \brack n}_q &= {2n\brack n}_q - q{2n\brack n+1}_q 
\end{align}
and Theorem 1 now follows. \qed

\section{Combinatorial discussion}

Since Theorem 1 is a statement specifically about northeastern lattice paths, it would be nice to find a proof that lives solely within the context of lattice paths. In this section, we provide a combinatorial proof of Theorem 1, along with a discussion of how Theorem 1 can be interpreted under other, equivalent reformulations of the problem.

\subsection{A combinatorial proof of Theorem 1}

The central idea to this proof is that we introduce a map on the set of northeastern lattice paths from $(0,0)$ to $(n,n)$ such that the map has a period of $2n - 1$ and cycles the area enclosed under each path modulo $2n - 1$ by an integer relatively prime to $2n - 1$. Such a map would partition the set of ${2n}\choose{n}$ lattice paths into disjoint classes of size $2n - 1$ with every possible area modulo $2n - 1$ being represented exactly once.

For the sake of convenience, let $P(n)$ refer to the set of northeastern lattice paths from $(0,0)$ to $(n,n)$. Every path $\pi \in P(n)$ can be uniquely described by a sequence of $2n$ unit steps $(\pi_{1}, \ldots, \pi_{2n}),$ where each $\pi_k$ is either the horizontal step $(1, 0)$ or the vertical step $(0, 1)$. An important distinction is between paths $\pi \in P(n)$ such that $\pi_{1} = (1, 0)$, we will call this set $P_{0}(n)$, and those such that $\pi_{1} = (0, 1)$, we will call this set $P_{1}(n)$. For now, we are only concerned with paths in $P_{0}(n)$, and we will deal with the other type of path later. 

We will now introduce a mapping $\phi$ between lattice paths $\pi \in P_{0}(n).$ Let $\phi(\pi) = \rho$, where

\[ \begin{cases} 
    \rho_{1} = \pi_{1} \\
    \rho_{2} = \pi_{2n} \\
    \rho_{k} = \pi_{k - 1}, & 3 \leq k \leq 2n.
   \end{cases}
\]
Figure \ref{map} portrays two examples of $\phi$ acting on a 4 by 4 lattice path.

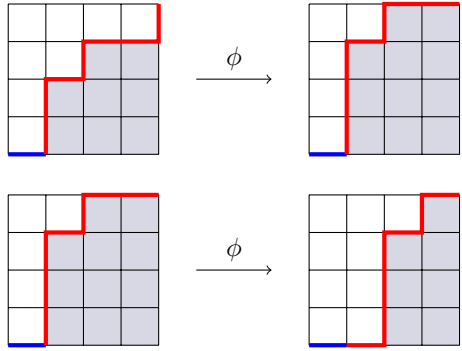
\begin{figure}[h]
\begin{center}
    \begin{tikzpicture}[scale=.5]
    \fill[blue!50!olive,opacity=.2](1,0)--(1,2)--(2,2)--(2,3)--(4,3)--(4,0)--cycle;
        \draw[step=1,thin] (0,0) node[anchor=north east] {} grid (4,4)node[anchor=south west] {};
        \draw[ultra thick,blue] (0,0) -- (1,0);
        \draw[ultra thick,red] (1,0) --(1,2)--(2,2)--(2,3)--(4,3)--(4,4);
        
        \begin{scope}[xshift=8cm]
        \fill[blue!50!olive,opacity=.2](1,0)--(1,3)--(2,3)--(2,4)--(4,4)--(4,0)--cycle;
        \draw[step=1,thin] (0,0) node[anchor=north east] {} grid (4,4)node[anchor=south west] {};
        \draw[ultra thick,blue] (0,0) -- (1,0);
        \draw[ultra thick,red] (1,0) --(1,3)--(2,3)--(2,4)--(4,4);
        \end{scope}
        \draw[->] (5,2)--(7,2);
        
        \node[anchor=south] at (6,2){$\phi$};
        
        \begin{scope}[yshift=-2in]
        \fill[blue!50!olive,opacity=.2](1,0)--(1,3)--(2,3)--(2,4)--(4,4)--(4,0)--cycle;
        \draw[step=1,thin] (0,0) node[anchor=north east] {} grid (4,4)node[anchor=south west] {};
        \draw[ultra thick,blue] (0,0) -- (1,0);
        \draw[ultra thick,red] (1,0) --(1,3)--(2,3)--(2,4)--(4,4);
        
        \begin{scope}[xshift=8cm]
        \fill[blue!50!olive,opacity=.2](2,0)--(2,3)--(3,3)--(3,4)--(4,4)--(4,0)--cycle;
        \draw[step=1,thin] (0,0) node[anchor=north east] {} grid (4,4)node[anchor=south west] {};
        \draw[ultra thick,blue] (0,0) -- (1,0);
        \draw[ultra thick,red] (1,0)--(2,0) --(2,3)--(3,3)--(3,4)--(4,4);
        \end{scope}
        \draw[->] (5,2)--(7,2);
        
        \node[anchor=south] at (6,2){$\phi$};
        \end{scope}
    \end{tikzpicture}
\end{center}
\caption{A graphical representation of the map $\phi$ on two paths in $P_0(4)$. Top: the enclosed area increases by 3. Bottom: the enclosed area decreases by 4. We have $3 \equiv -4$ (mod $7$), so the area changes by the same amount modulo 7.}
\label{map}
\end{figure}

This map can be interpreted visually as shifting the whole path after $(1,0)$ either north one unit or east one unit according to the value of $\pi_{2n}.$ It should be apparent from the definition of $\phi$ that it is cyclical and repeats every $2n - 1$ iterations, meaning that its period divides $2n - 1$.

We will now examine how $\phi$ affects the enclosed area of the lattice paths that it acts on. Let $\pi\in P_0(n)$ and suppose $\pi$ encloses an area of $k$ modulo $2n-1$. Notice that visually there are two ways that $\phi$ can affect the area of $\pi$: either $\phi$ ``deletes'' the rightmost column of enclosed unit squares, or $\phi$ ``inserts'' a bottom row of enclosed unit squares. The former occurs when $\pi_{2n} = (1,0)$, in which case the enclosed area of $\phi(\pi)$ decreases by $n$. The latter occurs when $\pi_{2n} = (0,1)$, in which case the enclosed area of $\phi(\pi)$ increases by $n-1$. However, modulo $2n - 1$, an increase by $n - 1$ is equivalent to a decrease by $n$. Thus, $\phi(\pi)$ encloses an area congruent to $k + n - 1$ modulo $2n - 1$. Since $n - 1$ is relatively prime to $2n - 1$, this consequently forces the period of $\phi$ to equal $2n - 1$. Figure \ref{mapall} demonstrates the effect that repeatedly applying $\phi$ on a path has on the enclosed area.

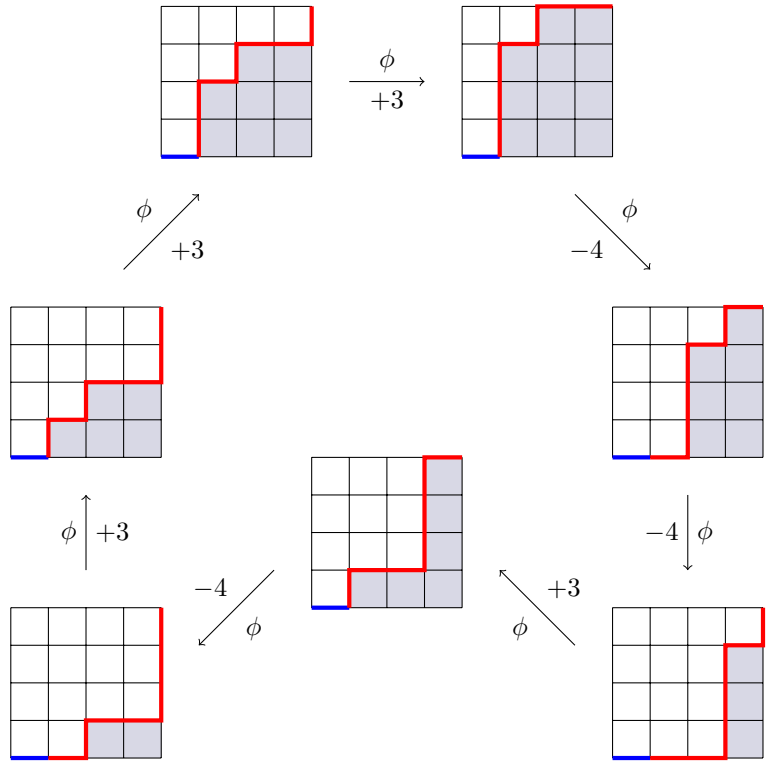
\begin{figure}[h]
\begin{center}
    \begin{tikzpicture}[scale=.5]
    \fill[blue!50!olive,opacity=.2](1,0)--(1,2)--(2,2)--(2,3)--(4,3)--(4,0)--cycle;
        \draw[step=1,thin] (0,0) node[anchor=north east] {} grid (4,4)node[anchor=south west] {};
        \draw[ultra thick,blue] (0,0) -- (1,0);
        \draw[ultra thick,red] (1,0) --(1,2)--(2,2)--(2,3)--(4,3)--(4,4);
        
        \begin{scope}[xshift=8cm]
        \fill[blue!50!olive,opacity=.2](1,0)--(1,3)--(2,3)--(2,4)--(4,4)--(4,0)--cycle;
        \draw[step=1,thin] (0,0) node[anchor=north east] {} grid (4,4)node[anchor=south west] {};
        \draw[ultra thick,blue] (0,0) -- (1,0);
        \draw[ultra thick,red] (1,0) --(1,3)--(2,3)--(2,4)--(4,4);
        \end{scope}
        
        \begin{scope}[xshift=12cm,yshift=-8cm]
        \fill[blue!50!olive,opacity=.2](2,0)--(2,3)--(3,3)--(3,4)--(4,4)--(4,0)--cycle;
        \draw[step=1,thin] (0,0) node[anchor=north east] {} grid (4,4)node[anchor=south west] {};
        \draw[ultra thick,blue] (0,0) -- (1,0);
        \draw[ultra thick,red] (1,0) --(2,0)--(2,3)--(3,3)--(3,4)--(4,4);
        \end{scope}
        
        \begin{scope}[xshift=12cm,yshift=-16cm]
        \fill[blue!50!olive,opacity=.2](3,0)--(3,3)--(4,3)--(4,0)--cycle;
        \draw[step=1,thin] (0,0) node[anchor=north east] {} grid (4,4)node[anchor=south west] {};
        \draw[ultra thick,blue] (0,0) -- (1,0);
        \draw[ultra thick,red] (1,0) --(3,0)--(3,3)--(4,3)--(4,4);
        \end{scope}
        
        \begin{scope}[xshift=4cm,yshift=-12cm]
        \fill[blue!50!olive,opacity=.2](1,0)--(1,1)--(3,1)--(3,4)--(4,4)--(4,0)--cycle;
        \draw[step=1,thin] (0,0) node[anchor=north east] {} grid (4,4)node[anchor=south west] {};
        \draw[ultra thick,blue] (0,0) -- (1,0);
        \draw[ultra thick,red] (1,0) --(1,1)--(3,1)--(3,4)--(4,4);
        \end{scope}
        
        \begin{scope}[xshift=-4cm,yshift=-16cm]
        \fill[blue!50!olive,opacity=.2](2,0)--(2,1)--(4,1)--(4,0)--cycle;
        \draw[step=1,thin] (0,0) node[anchor=north east] {} grid (4,4)node[anchor=south west] {};
        \draw[ultra thick,blue] (0,0) -- (1,0);
        \draw[ultra thick,red] (1,0) --(2,0)--(2,1)--(4,1)--(4,4);
        \end{scope}
        
        \begin{scope}[xshift=-4cm,yshift=-8cm]
        \fill[blue!50!olive,opacity=.2](1,0)--(1,1)--(2,1)--(2,2)--(4,2)--(4,0)--cycle;
        \draw[step=1,thin] (0,0) node[anchor=north east] {} grid (4,4)node[anchor=south west] {};
        \draw[ultra thick,blue] (0,0) -- (1,0);
        \draw[ultra thick,red] (1,0) --(1,1)--(2,1)--(2,2)--(4,2)--(4,4);
        \end{scope}

        \draw[->] (5,2)--(7,2);
        \node[anchor=south] at (6,2){$\phi$};
        \node[anchor=north] at (6,2){$+3$};
        
        \draw[->] (11,-1)--(13,-3);
        \node[anchor=south west] at (12,-2){$\phi$};
        \node[anchor=north east] at (12,-2){$-4$};
        
        \draw[->] (14,-9)--(14,-11);
        \node[anchor=west] at (14,-10){$\phi$};
        \node[anchor=east] at (14,-10){$-4$};
        
        \draw[->] (11,-13)--(9,-11);
        \node[anchor=north east] at (10,-12){$\phi$};
        \node[anchor=south west] at (10,-12){$+3$};
        
        \draw[->] (3,-11)--(1,-13);
        \node[anchor=north west] at (2,-12){$\phi$};
        \node[anchor=south east] at (2,-12){$-4$};
        
        \draw[->] (-2,-11)--(-2,-9);
        \node[anchor=east] at (-2,-10){$\phi$};
        \node[anchor=west] at (-2,-10){$+3$};
        
        \draw[->] (-1,-3)--(1,-1);
        \node[anchor=south east] at (0,-2){$\phi$};
        \node[anchor=north west] at (0,-2){$+3$};

    \end{tikzpicture}
\end{center}
\caption{Starting with a northeastern lattice path $\pi\in P_0(4)$ the enclosed area increases by 3 modulo 7 each time $\phi$ is applied. We thus obtain a set of 7 paths, each enclosing a different area modulo 7.}
\label{mapall}
\end{figure}

We now extend the domain of $\phi$ to also include paths in $P_1(n)$. Define the transpose of a lattice path $\pi$ as $\pi^{T} = ((\pi_{1,1}, \pi_{1,0}), \ldots, (\pi_{2n,1}, \pi_{2n,0}))$, which geometrically corresponds to mirroring the path along the main diagonal. Then we define $\phi(\pi) = \phi(\pi^{T})^{T}$ for all $\pi \in P_{1}(n).$ In order to check that this extended definition of $\phi$ harbors the same area-preserving property, let $\pi$ be an element of $P_1(n)$ and suppose that $\pi$ encloses an area congruent to $k$ modulo $2n - 1$. Then, $\pi^T$ encloses an area congruent to $n^2 - k$ modulo $2n - 1$, and since $\pi^T \in P_0(n)$, we know that $\phi(\pi^T)$ encloses an area congruent to $n^2 - k + n - 1$ modulo $2n - 1$. It follows that $\phi(\pi)$ encloses an area congruent to $n^2 - (n^2 - k + n - 1) = k - (n - 1) \equiv k + n$ modulo $2n - 1$. Since $n$ is also relatively prime to $2n - 1$, $\phi$ will still have period $2n - 1$ when acting on paths in $P_1(n)$. Therefore, the orbits of $\phi$ partition the set of northeastern lattice paths from $(0, 0)$ to $(n, n)$ into disjoint classes of size $2n - 1$ such that within every orbit of $\phi$, each lattice path will have a distinct enclosed area modulo $2n - 1$.

Suppose now that we want to count the number of northeastern lattice paths that enclose an area congruent to $k$ modulo $2n - 1$. We know that the total number of northeastern lattice paths confined to this grid is ${2n}\choose{n}$, and we found a way to divide these lattice paths evenly into sets of cardinality $2n - 1$, such that an enclosed area congruent to $k$ modulo $2n - 1$ occurs exactly once per set. We therefore conclude that the number of such paths is equal to $\frac{1}{2n - 1}{{2n}\choose{n}}$. \qed

\subsection{Binary words}

It is elementary to see that ${2n}\choose{n}$ equals the number of ways to choose $n$ elements from a set of size $2n$, thus looking at binary words is arguably the most natural way to study this binomial coefficient. A binary word is simply a finite sequence of 1's and 0's. Therefore, the number of binary words with length $2n$ such that exactly $n$ of the digits are 0 is ${2n}\choose{n}$. We will refer to such binary words as \textit{even binary words of length $2n$}.

The fact that $2n - 1$ divides ${2n}\choose{n}$ is easy to see in this context. Define the following map: for each even binary word $b$ of length $2n$, fix the first digit of $b$, and permute the other $2n - 1$ digits by shifting them one space to the right, moving the rightmost digit to the second position from the left. The part of $b$ being cycled has length $2n - 1$, so the period of this map divides $2n-1$. Furthermore, the number of 0's in this part is either equal to $n$ or $n - 1$. Since both $n$ and $n - 1$ are relatively prime to $2n - 1$, this map must have a period of exactly $2n - 1$. (More generally, a similar map shows that $n$ divides ${n}\choose{k}$ when $n$ and $k$ are relatively prime---we will explore this more in Section 4.)

However, this fact still does not illuminate what $\frac{1}{2n - 1}{{2n}\choose{n}}$ equals in this situation. To do so, we must explore the connection between northeastern lattice paths from $(0,0)$ to $(n,n)$ and even binary words of length $2n$. Notice that every even binary word $b$ with length $2n$ can actually be interpreted as an instruction for drawing a path on a lattice from $(0,0)$ to $(n, n)$. Start at $(0,0)$, and for each digit in $b$, move up one space if the digit is a 1 and move right one space if the digit is a 0. Since $b$ is even and contains $n$ ones and $n$ zeros the finishing coordinate will be $(n,n)$. Thus, the problem of binary words can simply be viewed as a reformulation of the problem of lattice paths. 

The area enclosed by a lattice path also has a natural interpretation in the context of binary words. An \textit{inversion} in a binary word $b$ is defined as a subsequence in $b$ of the form ``1  0''. We stress here that we do not mean adjacent subsequences; there are ${2n\choose 2}$ total subsequences. Let $I(b)$ denote the number of inversions in $b$. To count $I(b)$, simply take every 1 in word $b$ and count the number of 0's that come after that 1. Let $\pi$ denote the lattice path represented by $b$. The key observation is that every unit square enclosed by $\pi$ can be represented by its vertical height along with its horizontal distance from the lattice path at that height. Since every 1 in $b$ represents a distinct height along the path $\pi$, and since every 0 following that 1 represents a distinct horizontal distance away from the path at that height, every unit square enclosed by $\pi$ corresponds uniquely to an inversion in $b$. Therefore, the area enclosed by $\pi$ is equal to $I(b)$. Figure \ref{invpath} shows this connection.

\begin{figure}[h]
\begin{center}
    \begin{tikzpicture}[scale=.5]
    \fill[blue!50!olive,opacity=.2](1,0)--(1,2)--(2,2)--(2,3)--(4,3)--(4,0)--cycle;
    \fill[blue!50!olive,opacity=.4] (3,1) rectangle (4,2);
        \draw[step=1,thin] (0,0) node[anchor=north east] {} grid (4,4)node[anchor=south west] {};
        \draw[ultra thick,red] (0,0) -- (1,0);
        \draw[ultra thick,red] (1,0) --(1,2)--(2,2)--(2,3)--(4,3)--(4,4);
        
        \node at (9,2){$0\ 1\ \underline{1}\ 0\ 1\ 0\ \underline{0}\ 1$};

    \end{tikzpicture}
\end{center}
\caption{A northeastern lattice path from $(0,0)$ to $(4,4)$ enclosing an area of 8 (left) and the corresponding even binary word of length 8 having 8 inversions (right). The underlined inversion corresponds to the highlighted box.}
\label{invpath}
\end{figure}
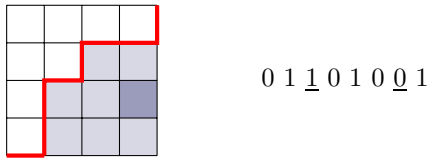

We can now restate Theorem 1 in terms of binary words.

\begin{thm}
    The number of even binary words of length $2n$ having the number of inversions congruent to $k$ modulo $2n - 1$ is independent of $k$, and is equal to $\frac{1}{2n - 1}{{2n}\choose{n}}$.
\end{thm}

\begin{rmk}
    There is a connection between inversions in a binary word $b$ and another statistic, called the \textit{major index} of $b$, which is the sum of the indices of $b$ for which there is a 1 followed immediately by a 0. For example, \[{\bf1} \ 0 \ 0 \ 1 \ {\bf1} \ 0 \ 1\] has major index $1+5=6$. A result of MacMahon \cite{MM} implies the number of even binary word of length $2n$ with $k$ inversions equals the number of even binary words of length $2n$ with major index equal to $k$. In light of this, we may replace ``the number of inversions" with ``major index" in Theorem 2.
\end{rmk}

\subsection{Increasing integer sequences}

Now ${2n}\choose{n}$ also equals the number of increasing integer sequences of the form $1 \leq a_{1} < \ldots < a_{n} \leq 2n$, since we are choosing $n$ distinct numbers out of $2n$ total numbers. It would then be natural to ask what the value $\frac{1}{2n - 1}{{2n}\choose{n}}$ represents in terms of increasing integer sequences.

\begin{thm}
    The number of increasing integer sequences of length $n$ bounded by 1 and $2n$ that have a sum congruent to $k$ modulo $2n - 1$ is independent of $k$, and is equal to $\frac{1}{2n - 1}{{2n}\choose{n}}$.
\end{thm}

The proof of this claim is very much of the same flavor as the proof given in Section 3.1. Let $S(n) = \left \{\alpha= \{a_{1}, \ldots, a_{n}\} : 1 \leq a_{1} < \ldots < a_{n} \leq 2n \right \}$ denote the set of increasing integer sequences of length $n$ that are bounded by $1$ and $2n$. We will make a distinction between sequences $\alpha \in S(n)$ such that $a_{1} = 1$ which we will call $S_{0}(n)$, and sequences such that $a_{1} > 1$ which we will call $S_{1}(n)$. 

We now define a couple of maps. Let $\phi_{0}$ be a map on $S_{0}(n).$ For all $\alpha \in S_{0}(n)$, let $\phi_{0}(a) = \beta=\{b_1,b_2\ldots,b_n\}$, where

\[ \begin{cases} 
    b_{1} = a_{1}, \\
    b_{k} = a_{k} + 1 & \text{for }2 \leq k \leq n - 1, \\
    b_{n} = a_{n} + 1 & \text{if } a_{n} < 2n, \\
    b_{n} = 2, & \text{if } a_{n} = 2n.
   \end{cases}
\]

Notice that this map was constructed to repeat every $2n - 1$ iterations. It remains to determine how applying $\phi_{0}$ on $\alpha$ affects the sum of its terms. Note that either $b_{n} = a_{n} + 1$ or $b_{n} = 2 = a_{n} - (2n - 2)$, and in either case $b_{n} \equiv a_{n} + 1 \ (\textrm{mod } 2n - 1)$. Thus,

\[
    \sum_{k = 1}^{n}b_{k} = b_{1} + \sum_{k = 2}^{n}b_{k} \\
    \equiv a_{1} + \sum_{k = 2}^{n} (a_{k} + 1) \\
    \equiv \sum_{k = 1}^{n} a_{k} + (n - 1)\ (\textrm{mod } 2n - 1).
\]

Similarly, let $\phi_{1}$ be a map on $S_{1}(n).$ For all $\alpha \in S_{1}(n)$, let $\phi_{1}(\alpha) = \beta$, where 

\[ \begin{cases} 
    b_{k} = a_{k} + 1 & \text{for } 1 \leq k \leq n - 1, \\
    b_{n} = a_{n} + 1 & \text{if } a_{n} < 2n, \\
    b_{n} = 2 & a_{n} = \text{if }2n.
   \end{cases}
\]

This map also repeats every $2n - 1$ iterations, and similar reasoning can be applied to show that applying $\phi_{1}$ on $\alpha$ increases the sum of its terms by $n$ modulo $2n - 1$. 

Finally, we will define a map $\phi^{\prime}$ on $S(n)$ as follows:

\[ \phi^{\prime}(\alpha) = \begin{cases}
        \phi_{0}(\alpha), & \alpha \in S_{0} \\
        \phi_{1}(\alpha), & \alpha \in S_{1}.
        \end{cases}
\]

Now $\phi^{\prime}$ splits $S(n)$ up into disjoint classes of size $2n - 1$. Since $n$ and $n - 1$ are relatively prime to $2n - 1$, every class contains $2n - 1$ distinct sums modulo $2n - 1$. Thus, the number of occurrences of any particular sum modulo $2n - 1$ is equal to $\frac{1}{2n - 1}{{2n}\choose{n}}$.

\begin{rmk}
    We may also consider {\it weakly} increasing sequences of the form $0\leq a_1\leq a_2 \leq \ldots \leq n$. These weakly increasing sequences can be thought of as integer partitions with at most $n$ parts, each of which is at most $n$. There is a one-to-one correspondence between these weakly increasing sequences and northeastern lattice paths from $(0,0)$ to $(n,n)$, where $a_i$ gives the enclosed area contained in the $i$th column of the $n\times n$ grid. Thus Theorem 3 can be restated in terms of the weakly increasing sequences described above.
\end{rmk}

\section{Generalizations and Applications}

In this section we will examine how our combinatorial proof of Theorem 1 in Section 3.1 adapts when looking at divisors of other binomial coefficients. 

\subsection{When $n$ and $k$ are relatively prime}

It is well known that if $p$ is prime then for any $0<k<p$ the binomial coefficient ${p\choose k}$ is divisible by $p$. A slight generalization of this fact states that if $0<k<n$ and $n$ and $k$ are relatively prime then $n$ divides ${n\choose k}$. This result follows from Kummer's theorem \cite{kummer_e_e_1852_1448864} that states the $p$-adic valuation of ${n\choose k}$ is the number of carries when $k$ and $n-k$ are added in their base-$p$ representation. 

We now fix $0<k<n$ with $n$ and $k$ relatively prime. We will show that our map $\phi$ described in the combinatorial proof of Theorem 1 of Section 3.1 can be adapted to show the following.

\begin{thm}
    If $n$ and $k$ are relatively prime with $0<k<n$, then for any $0\leq m < n$ the number of northeastern lattice paths from $(0,0)$ to $(k,n-k)$ enclosing an area of $m$ modulo $n$ is independent of $m$, and equal to \[\frac{1}{n}{n\choose k}.\]
\end{thm}

Just as in Section 3.1 we write a path $\pi$ as a sequence of $n$ coordinate pairs $(\pi_1,\ldots,\pi_n)$ as well. Then we define $\phi$ by setting $\phi(\pi)=\rho$, where 

\[\rho_i = \begin{cases}
\pi_{n} & i=1, \\
\pi_{i-1} & 2\leq i \leq n .
\end{cases}
\]

For example, take the pair $(n,k)=(8,5)$. Then the action of the map $\phi$ on a path $\pi$ can be visualized in Figure \ref{ntok}.

\begin{figure}[h]
\begin{center}
    \begin{tikzpicture}[scale=.5]
    \fill[blue!50!olive,opacity=.2](0,0)--(0,1)--(2,1)--(2,2)--(4,2)--(4,3)--(5,3)--(5,0)--cycle;
        \draw[step=1,thin] (0,0) node[anchor=north east] {} grid (5,3)node[anchor=south west] {};
        \draw[ultra thick,red](0,0)-- (0,1) --(2,1)--(2,2)--(4,2)--(4,3)--(5,3);
        
        \begin{scope}[xshift=9cm]
        \fill[blue!50!olive,opacity=.2](1,0)--(1,1)--(3,1)--(3,2)--(5,2)--(5,0)--cycle;
        \draw[step=1,thin] (0,0) node[anchor=north east] {} grid (5,3)node[anchor=south west] {};
        \draw[ultra thick,red] (0,0)--(1,0) --(1,1)--(3,1)--(3,2)--(5,2)--(5,3);
        \end{scope}
        \draw[->] (6,1.5)--(8,1.5);
        
        \node[anchor=south] at (7,1.5){$\phi$};
    \end{tikzpicture}
\end{center}
\caption{The action of the map $\phi$ on one of the ${8\choose 5}=56$ northeastern lattice paths from $(0,0)$ to $(8,5)$.}
\label{ntok}
\end{figure}
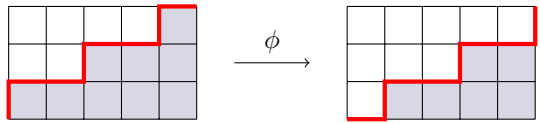

Notice that every time $\phi$ is applied to a path $\pi$ the area enclosed by $\pi$ is either increased by $k$ or reduced by $n-k$. But modulo $n$ this is increasing the enclosed area by $k$, and since $k$ is relatively prime to $n$ the map $\phi$ has period $n$. Thus we are partitioning the set of all ${n\choose k}$ paths into sets of $n$ paths, with each path in a set enclosing a different area modulo $n$.

\begin{rmk}
It should be mentioned that this result also follows from Lemma 1 and the fact that ${n\brack k}_q$ can be shown to be divisible by $[n]_q$ when $n$ and $k$ are relatively prime.
\end{rmk}

Now we observe that Theorem 1 is in some sense a special case of this more general observation. Since \[{2n \choose n} = 2\cdot {2n-1\choose n}\] and $2n-1$ and $n$ are relatively prime, Theorem 1 immediately follows.

\subsection{Integer products}

We return to looking at increasing integer sequences $1 \leq a_{1} < \ldots < a_{k} \leq n$. When $n$ and $k$ are relatively prime the sequences will be equally distributed according to their sum modulo $n$. Under certain conditions a similar statement can be made about the distribution of these sequences by their products modulo $n$.

\begin{thm}
    Let $p$ be a prime and $\ell$ be a number between $1$ and $p-1$ that is relatively prime to $p-1$. For $1\leq k \leq p-1$ the number of increasing integer sequences of length $\ell$ bounded by $1$ and $p-1$ that have product congruent to $k$ modulo $p$ is independent of $k$, and equal to $\frac{1}{p-1}{p-1\choose \ell}$.
\end{thm}

The proof of this claim follows from the analogous result for integer sums and the fact that the unit group of the integers modulo $p$ is cyclic and isomorphic to the group of integers modulo $p-1$.

Let $1\leq a_1 < \ldots < a_{\ell} \leq p-1$ be an increasing integer sequence. Because the unit group of integers modulo $p$ is cyclic, there is a number $a$ such that $a_i = a^{j_i}$ for each $i$. Furthermore each $j_i$ is distinct and bounded between $0$ and $p-1$. Because $\ell$ is relatively prime to $p-1$, the sums of the $j_i$ are equally distributed in the residue classes modulo $p-1$. Therefore the $\frac{1}{p-1}{p-1\choose \ell}$ sums $\sum j_i$ congruent to $k^*$ modulo $p-1$ correspond to the same number of products $\prod a_i$ congruent to $k=a^{k^*}$ modulo $p-1$. 

\subsection{When $n$ and $k$ are not relatively prime}

While the condition that $n$ and $k$ be relatively prime is sufficient to conclude that $n\mid {n\choose k}$, it is not necessary. The first example of this is when $n=10$ and $k=4$. Clearly 10 and 4 are not relatively prime, however $10\mid {10\choose 4}=210$. We then ask the question how the 210 northeastern lattice paths from $(0,0)$ to $(6,4)$ are distributed by their enclosed area modulo 10.

\begin{table}[h]
    \begin{center}
    \begin{tabular}{c|c||c|c}
         Enclosed area modulo 10&\# of paths&Enclosed area modulo 10&\# of paths  \\
         \hline 
         0&22&
         5&20\\
         1&20&
         6&22\\
         2&22&
         7&20\\
         3&20&
         8&22\\
         4&22&
         9&20
    \end{tabular}
    \end{center}
    \caption{The distribution of northeastern lattice paths from $(0,0)$ to $(6,4)$ by enclosed area modulo 10.}
    \label{10tab}
\end{table}

As seen in Table \ref{10tab}, the paths do not fall evenly into residue classes modulo 10, but they do fall evenly into residue classes modulo 5. We can give a quick combinatorial reason for why this is the case. When $n$ and $k$ were relatively prime, since a single application of $\phi$ on a path from $(0,0)$ to $(k,n-k)$ increased the enclosed area by $k$ modulo $n$, repeated application of $\phi$ gave us a set of paths distributed evenly among all residue classes modulo $n$. Now if $g>1$ is the greatest common divisor of $n$ and $k$, a single application of $\phi$ on a path still increases the area by $k$ modulo $n$; however, after repeated application we end up with a set of paths distributed evenly amongst all residue classes modulo $n/g$. We state this as a theorem.

\begin{thm}
     Let $g$ be the greatest common divisor of $n$ and $k$, where $0<k<n$. Then for any $0\leq j < g$, the number of northeastern lattice paths from $(0,0)$ to $(k,n-k)$ enclosing an area of $j$ modulo $g$ is independent of $j$, and equal to \[\frac{g}{n}{n\choose k}.\]
\end{thm}

\begin{rmk}

This theorem has an interesting application to polynomial divisors of the Gaussian binomial coefficients. It is tempting to assume that if $d\mid {n\choose k}$ then it must be the case that $[d]_q\mid {n\brack k}_q$ as elements of $\mathbb{Z}[q]$. But by our observation regarding the northeastern lattice paths from $(0,0)$ to $(6,4)$ and Lemma 1, we see that this is false. What we are able to say is that if $n\mid {n\choose k}$ and $g$ is the greatest common divisor of $n$ and $k$ then $[n/g]_q\mid {n\brack k}_q$ as elements of $\mathbb{Z}[q]$.

\end{rmk}

\subsection{A note on the Catalan numbers}

Our motivation for this study was the Catalan numbers, for which the Chung--Feller theorem gives a combinatorial description for the divisibility of ${2n\choose n}$ by $n+1$. We now show that our map $\phi$ described in the combinatorial proof of Theorem 1 in Section 3.1 can be adapted to give a seemingly new combinatorial explanation for why the Catalan numbers are integers.

\begin{thm}
Let $n\geq 1$ and $0\leq j \leq n$. Then the number of northeastern lattice paths from $(0,0)$ to $(n,n)$ enclosing an area congruent to $j$ modulo $n+1$ is equal to $C_n = \frac{1}{n+1}{2n\choose n}$, which in particular is independent from $j$.
\end{thm}
We begin by considering the ${2n+1\choose n}$ northeastern lattice paths from $(0,0)$ to $(n+1,n)$. The identity illuminating this approach is the alternative definition of the Catalan numbers \[C_n = \frac{1}{n+1}{2n\choose n} = \frac{1}{2n+1}{2n+1\choose n}.\]
Now by applying our map $\phi$ described in Section 3.1 to these paths we partition the set of paths into $C_n$ equal parts, each of size $2n+1$. In each one of these parts there is exactly one path enclosing an area congruent to $k$ modulo $2n+1$ for each $0\leq k\leq 2n$. Of these $2n+1$ paths, in one of the parts of the partition exactly $n+1$ can also be viewed as northeastern lattice paths from $(0,0)$ to $(n,n)$---as shown in Figure \ref{wow}.

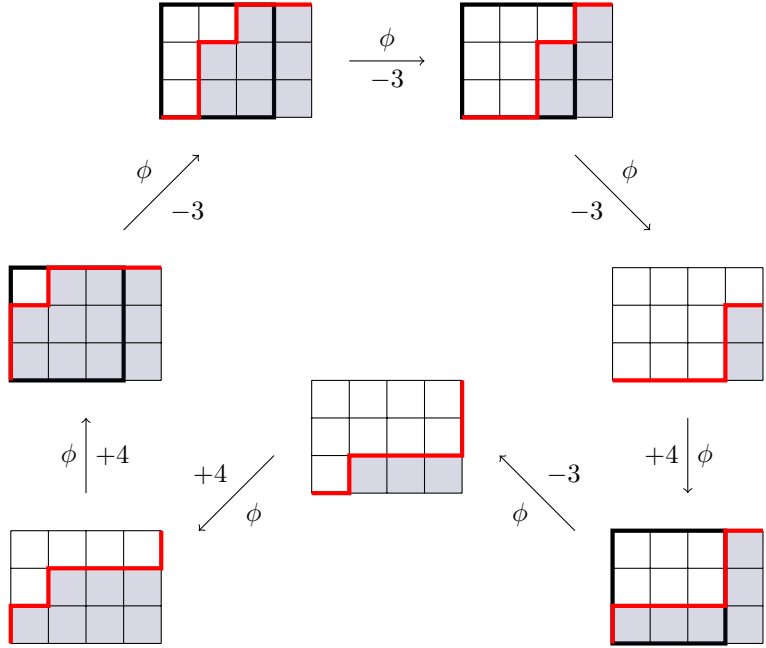
\begin{figure}[h]
\begin{center}
    \begin{tikzpicture}[scale=.5]
    \draw[ultra thick] (0,0)--(3,0)--(3,3)--(0,3)--cycle;
    \fill[blue!50!olive,opacity=.2](1,0)--(1,2)--(2,2)--(2,3)--(4,3)--(4,0)--cycle;
        \draw[step=1,thin] (0,0) node[anchor=north east] {} grid (4,3)node[anchor=south west] {};
        \draw[ultra thick,red] (0,0)--(1,0) --(1,2)--(2,2)--(2,3)--(4,3);
        
        \begin{scope}[xshift=8cm]
        \draw[ultra thick] (0,0)--(3,0)--(3,3)--(0,3)--cycle;
        \fill[blue!50!olive,opacity=.2](2,0)--(2,2)--(3,2)--(3,3)--(4,3)--(4,0)--cycle;
        \draw[step=1,thin] (0,0) node[anchor=north east] {} grid (4,3)node[anchor=south west] {};
        \draw[ultra thick,red] (0,0)--(2,0) --(2,2)--(3,2)--(3,3)--(4,3);
        \end{scope}
        
        \begin{scope}[xshift=12cm,yshift=-7cm]
        \fill[blue!50!olive,opacity=.2](3,0)--(3,2)--(4,2)--(4,0)--cycle;
        \draw[step=1,thin] (0,0) node[anchor=north east] {} grid (4,3)node[anchor=south west] {};
        \draw[ultra thick,red] (0,0)--(3,0) --(3,2)--(4,2);
        \end{scope}
        
        \begin{scope}[xshift=12cm,yshift=-14cm]
        \draw[ultra thick] (0,0)--(3,0)--(3,3)--(0,3)--cycle;
        \fill[blue!50!olive,opacity=.2](0,0)--(0,1)--(3,1)--(3,3)--(4,3)--(4,0)--cycle;
        \draw[step=1,thin] (0,0) node[anchor=north east] {} grid (4,3)node[anchor=south west] {};
        \draw[ultra thick,red] (0,0)--(0,1) --(3,1)--(3,3)--(4,3);
        \end{scope}
        
        \begin{scope}[xshift=4cm,yshift=-10cm]
        \fill[blue!50!olive,opacity=.2](1,0)--(1,1)--(4,1)--(4,0)--cycle;
        \draw[step=1,thin] (0,0) node[anchor=north east] {} grid (4,3)node[anchor=south west] {};
        \draw[ultra thick,red] (0,0)--(1,0) --(1,1)--(4,1)--(4,3);
        \end{scope}
        
        \begin{scope}[xshift=-4cm,yshift=-14cm]
        \fill[blue!50!olive,opacity=.2](0,0)--(0,1)--(1,1)--(1,2)--(4,2)--(4,0)--cycle;
        \draw[step=1,thin] (0,0) node[anchor=north east] {} grid (4,3)node[anchor=south west] {};
        \draw[ultra thick,red] (0,0)--(0,1) --(1,1)--(1,2)--(4,2)--(4,3);
        \end{scope}
        
        \begin{scope}[xshift=-4cm,yshift=-7cm]
        \draw[ultra thick] (0,0)--(3,0)--(3,3)--(0,3)--cycle;
        \fill[blue!50!olive,opacity=.2](0,0)--(0,2)--(1,2)--(1,3)--(4,3)--(4,0)--cycle;
        \draw[step=1,thin] (0,0) node[anchor=north east] {} grid (4,3)node[anchor=south west] {};

        \draw[ultra thick,red] (0,0)--(0,2) --(1,2)--(1,3)--(4,3);
        \end{scope}

        \draw[->] (5,1.5)--(7,1.5);
        \node[anchor=south] at (6,1.5){$\phi$};
        \node[anchor=north] at (6,1.5){$-3$};
        
        \draw[->] (11,-1)--(13,-3);
        \node[anchor=south west] at (12,-2){$\phi$};
        \node[anchor=north east] at (12,-2){$-3$};
        
        \draw[->] (14,-8)--(14,-10);
        \node[anchor=west] at (14,-9){$\phi$};
        \node[anchor=east] at (14,-9){$+4$};
        
        \draw[->] (11,-11)--(9,-9);
        \node[anchor=north east] at (10,-10){$\phi$};
        \node[anchor=south west] at (10,-10){$-3$};
        
        \draw[->] (3,-9)--(1,-11);
        \node[anchor=north west] at (2,-10){$\phi$};
        \node[anchor=south east] at (2,-10){$+4$};
        
        \draw[->] (-2,-10)--(-2,-8);
        \node[anchor=east] at (-2,-9){$\phi$};
        \node[anchor=west] at (-2,-9){$+4$};
        
        \draw[->] (-1,-3)--(1,-1);
        \node[anchor=south east] at (0,-2){$\phi$};
        \node[anchor=north west] at (0,-2){$-3$};

    \end{tikzpicture}
\end{center}
\caption{The action of $\phi$ on a northeastern lattice path from $(0,0)$ to $(4,3)$, corresponding to $n=3$. There are $2n+1=7$ paths in this cycle, however only $n+1=4$ of these paths pass through the point $(3,3)$.}
\label{wow}
\end{figure}

Thus the map $\phi$ can be restricted to the $n+1$ northeastern lattice paths from $(0,0)$ to $(n,n)$ by simply applying $\phi$ as many times as necessary. This restriction of $\phi$ clearly has period $n+1$, and each of these paths encloses a different area modulo $2n+1$. We now show that in fact each of these $n+1$ paths has a different area modulo $n+1$, so that we have a combinatorial description of the analogue of Theorem 1 for the divisor $n+1$ of ${2n\choose n}$.\footnote{A simple application of Lemma 1 and equation \ref{catd} also gives a proof of this fact.}

Starting with a path $\pi$ from $(0,0)$ to $(n+1,n)$ enclosing an area equal to $k$ that can be viewed as a path from $(0,0)$ to $(n,n)$, the path from $(0,0)$ to $(n,n)$ encloses an area of $k-n$. Let $j$ be the smallest positive integer such that $\phi^{(j)}(\pi)$ can be viewed as a path from $(0,0)$ to $(n,n)$. Now, applying $\phi$ $j$ times changes the area of the path from $(0,0)$ to $(n+1,n)$ by $(n+1)(j-1)-n$, since the path must first shift right one step and then shift up $j-1$ steps. Therefore, the new path from $(0,0)$ to $(n,n)$ encloses an area of $k+(n+1)(j-1)-2n$, which is a change of $(n+1)(j-1)-n$. Since $(n+1)(j-1)-n$ is congruent to 1 modulo $n+1$, this establishes the result.

\section*{Acknowledgments}
The first author (M.J.) was partially supported by the Research and Training Group grant DMS-1344994 funded by the National Science Foundation. We thank Robert Schneider, Elise Marchessault, and the referee for helpful comments.

\end{document}